\documentclass{siamltex}
\usepackage{enumerate}
\usepackage{amssymb}
\usepackage[normalem]{ulem}
\def\G{{\mathcal{G}}}
\def\bG{{\mathbb{G}}}
\def\h{{\mathcal{H}}}

\def\F{\mathcal{F}}
\def\bd{\mathbf{d}}
\def\bu{\mathbf{u}}
\def\bw{\mathbf{w}}
\def\bs{\mathbf{s}}
\def\mfw{\mathfrak{w}}
\def\mfW{\mathfrak{W}}
\def\mfu{\mathfrak{u}}
\def\mfU{\mathfrak{U}}
\def\mfE{\mathfrak{E}}
\def\mfG{\mathfrak{G}}
\def\mfS{\mathfrak{S}}
\def\mfd{\mathfrak{d}}
\def\mff{\mathfrak{f}}
\def\bbd{\mathbf{bd}}
\def\bdd{\mathbf{dd}}

\def\a{d^+}
\def\b{d^-}
\newtheorem{remark}[theorem]{Remark}
\newtheorem{conjecture}[theorem]{Conjecture}

\usepackage{color}
\newcommand{\chg}[1]{\textcolor{black}{ #1}}
\usepackage{url}

\begin{document}
\title{New classes of degree sequences with fast mixing swap Markov chain sampling}

\author{P\'eter L. Erd\H os\footnotemark[2]\ \footnotemark[5]
\and Istv\'an Mikl\'os\footnotemark[1]\footnotemark[2]\ \footnotemark[5] \footnotemark[7]
\and Zolt\'an Toroczkai\footnotemark[3] \footnotemark[8] }

\renewcommand{\thefootnote}{\fnsymbol{footnote}}
\footnotetext[1]{MTA SZTAKI, L\'agym\'anyosi \'ut 11, Budapest, 1111 Hungary}
\footnotetext[2]{MTA A. R{\'e}nyi Institute of Mathematics, Re\'altanoda u 13-15 Budapest, 1053 Hungary.
{\tt email:} \{erdos.peter, miklos.istvan\}@renyi.mta.hu}
\footnotetext[3]{Department of Physics and Interdisciplinary Center for Network Science \& Applications \\
University of Notre Dame, Notre Dame, IN, 46556, USA.
 {\tt email:} toro@nd.edu}
\footnotetext[5]{Supported partly by National Research, Development and Innovation Office – NKFIH, under the grants K 116769 and SNN 116095.}
\footnotetext[7]{Correspondence to: I. Mikl\'os}
\footnotetext[8]{Supported in part by the Defense Threat Reduction
Agency, \#HDTRA 1-09-1-0039 and jointly by the U.S. Air
Force Office of Scientific Research (AFOSR) and the
Defense Advanced Research Projects Agency (DARPA) under contract FA9550-12-1-0405. }
\renewcommand{\thefootnote}{\arabic{footnote}}

\maketitle

\begin{abstract}
In network modeling of complex systems one is often required to sample random realizations of networks that obey a given set of constraints, usually in form of graph measures.  A much studied class of problems targets uniform sampling of simple graphs with given degree sequence or also with given degree correlations expressed in the form of a joint degree matrix. One approach is to use Markov chains based on edge switches (swaps) that preserve the constraints, are irreducible (ergodic) and fast mixing. In 1999, Kannan, Tetali and Vempala (KTV) proposed a simple swap Markov chain for sampling graphs with given degree sequence and conjectured that it mixes rapidly (in poly-time) for arbitrary degree sequences. While the conjecture is still open, it was proven for special degree sequences, in particular, for those of undirected and directed regular simple graphs, of half-regular bipartite graphs, and of graphs with certain bounded maximum degrees. Here we prove the fast mixing KTV conjecture for novel, exponentially large classes of irregular degree sequences. Our method is based on a canonical decomposition of degree sequences into split graph degree sequences,  a structural theorem for the space of graph realizations and on a factorization theorem for Markov chains. After introducing bipartite splitted degree sequences, we also generalize the canonical split graph decomposition for bipartite and directed graphs.
 \end{abstract}



\pagestyle{myheadings}
\thispagestyle{plain}
\markboth{P.L. ERD\H{O}S, I. MIKL\'OS AND Z. TOROCZKAI}{DEGREE SEQUENCE DECOMPOSITION FOR FAST MIXING MCMC}

\section{Introduction}\label{sec:intro}

\noindent
Network science has been experiencing an explosive growth with applications in social sciences, economics, transportation infrastructures (energy and materials), communications, biology (from the molecular scale to that of species interactions), climate, and even in cosmology. An important problem in network science is to algorithmically  construct {\em typical} instances of the networks under study with predefined properties, often expressed as graph measures. In particular, special attention has been devoted to sampling simple graphs with a given degree sequence, both by the statistics (binary contingency tables  \cite{Bezakova},  \cite{R1}, \cite{R4}, \cite{DG95}, \cite{sonja1}, \cite{Gross}) and the computer science communities.  For relationships with algebraic statistics, see the survey paper by Petrovi\'c \cite{sonja2}.
Graph sampling methods can be classified roughly into two types, one using direct construction methods \cite{kim1,EMT09,Blitz,SCS,scdir} combined with importance sampling \cite{SCS,Blitz,scdir} and the other using simple edge-swap Markov chains  and corresponding Markov Chain Monte Carlo (MCMC)  algorithms \cite{R6}, \cite{R7}, \cite{R8}, \cite{R5}, \cite{rand}. Our focus here is on the latter, MCMC method.

In 1999 Kannan, Tetali and Vempala \cite{KTV97} (KTV) conjectured that a simple, edge-swap based Markov chain for sampling graphs at random with given degree sequence mixes rapidly, i.e., pseudo-random realizations can be obtained after polynomially many steps (polynomial in the length of the degree sequence, or order of the graph). For this Markov chain we start from an arbitrary graph realizing the degree sequence, then repeatedly draw uniformly at random, pairs of independent edges and swap their ends to create new realizations, as long as the swaps do not create multiple edges (if they do, we do not accept the new state, we simply draw again). This edge swap (also called 2-switch) operation, which clearly preserves the degree sequence, was studied by several authors before, including Ryser \cite{ryser}, Taylor \cite{taylor} and others \cite{rao}. The corresponding Markov chain is irreducible, aperiodic, reversible (obeys detailed balance), it has a symmetric transition matrix, and thus a uniform stationary distribution.

The first result with a correct proof in connection with the KTV conjecture is due to Cooper, Dyer and Greenhill (in 2007, \cite{CDG07}) for the special case when the degree sequence is regular. Greenhill then proved in 2011  the analogous result for (in- and out-)regular directed graphs \cite{G11}. In 2013 Mikl\'os, Erd\H{o}s and Soukup proved the conjecture for {\em half-regular} bipartite graphs \cite{MES}. Here the degree sequence on one side of the partition is regular, while the degrees can be arbitrary on the other side. Most recently, Greenhill proved the conjecture for simple graphs with relatively small maximal degrees, and also recently, Erd\H{o}s, Kiss, Mikl\'os and Soukup proved the conjecture for {\em almost half-regular} bipartite graphs with certain forbidden edge sets \cite{EKMS}. (Comprehensive surveys on the topic  can be found in \cite{G11} or \cite{MES}.)

These proofs are all based on the original Sinclair's multicommodity flow method \cite{S92} and they are rather technical. In the paper \cite{MCMC-JDM} we introduced an alternative approach to help prove the fast mixing nature of a restricted, edge-swap based MCMC over the balanced graphical realizations of a given {\em Joint Degree Matrix} (JDM for short). The word ``restricted" here refers to the fact, that not all traditional swap operations are allowed in the Markov Chain in order to preserve the given JDM; for details see \cite{JDM}.  Due to the special structure of balanced realizations of a JDM, being formed by a series of {\em almost-half-regular bipartite} degree sequences  and {\em almost-regular} degree sequences, one could exploit the previously obtained fast mixing results with the help of  a {\em general decomposition theorem} for Markov chains (\cite[Theorem 4.3]{MCMC-JDM}).

\medskip\noindent In this paper we follow a similar approach to extend the degree sequence classes with provable fast mixing Markov Chains. However, instead of using the above mentioned general, but somewhat involved chain decomposition result (or other, similar MCMC decomposition methods such as \cite{DSC93, MR02, MR06}), here we will employ a decomposition theorem of a  lesser generality, but one which is much easier to apply. Essentially, it is the statement that if the space of the Markov chain can be expressed as a Cartesian product of spaces such that the chain restricted over each one of the factor spaces is rapidly mixing, then it mixes rapidly over the whole space. This result \cite[Theorem 5.1]{MCMC-JDM} will be discussed in Section \ref{sec:tools}.

We will apply our methodology to two problem classes, namely, first to the original KTV conjecture itself and then second, to a similar problem related to sampling graphical realizations with given degree spectra \chg{\cite{JDM}}. In the first application we exploit the {\em canonical decomposition} of degree sequences (and of their realizations) into split degree sequences (and into split graphs, respectively), introduced by Tyshkevich (\cite{T80, T00}) and then the fact that the graph of all the graphical realizations of a degree sequence $\bd$ (the so-called realization graph $\bG(\bd)$) can be expressed as the {\em Cartesian product} of the realization graphs of the {\em factor} degree sequences from the canonical decomposition.
The later statement was proven recently by Barrus and West \cite{BW} and Barrus \cite{B15}. We will report these results in Section \ref{sec:decomp}. By exploiting a natural correspondence between split graphs and bipartite graphs, in Section \ref{sec:bipartite} we introduce the notion of splitted bipartite sequences (and their graphical realizations) and generalize these decomposition results for bipartite and directed graphical sequences as well. In Section \ref{sec:direct} we then apply our Markov chain decomposition theorem to show fast mixing for large classes of new degree sequences, bipartite, directed and undirected (also non-bipartite), constructed from composing splitted bipartite degree sequences with known fast mixing MCMC samplers. We then present estimates and comparisons for the sizes of these new degree sequence classes.

\medskip\noindent
The second application, of a smaller scope, is closely related to the JDM problem and it is a straightforward consequence of our method. In the paper \cite{JDM} the notion of {\em degree spectrum} was introduced as part of the solution for the  connectivity problem of the space of all graphical realizations of a given JDM (so that the corresponding MCMC is irreducible). Note that the JDM, which specifies the number of connections between given degrees, also uniquely determines the degree sequence and thus it is more constraining than just the degree sequence. In other words, there can be several JDMs with the same degree sequence.  The degree spectrum of a vertex $v$ is a vector of $\Delta(G)$ elements (the maximum degree in the realization $G$)  where the $i$th element is the number of degree $i$ neighbors of $v$. The {\em degree spectra matrix} $M$ of a graph $G$ contains the degree spectra of all its vertices as columns \cite{sampling-JDM}. The degree spectra matrix (DSM) is even more constraining than the JDM, as there can be several DSMs sharing the same JDM. Recently, Barrus and Donovan studied the same notion under a different name, called {\em neighborhood degree lists} \cite{BD15}, but for different reasons. In Section \ref{sec:spectra} we will discuss the degree spectra matrices in some detail and we present a class of DSMs with fast mixing MCMC samplers. This also implies that the corresponding JDMs  and therefore the corresponding degree sequences all admit fast MCMC samplers over the set of realizations {\em restricted to these degree spectra matrices}.

\section{Preliminaries}\label{sec:tools}

In this Section we list the common definitions, notations  and some of the earlier
results on MCMC sampling needed to present our findings.

\subsection{Graphs}\label{subsec:graphs}

Let us fix a labeled underlying vertex set $V$ of $n$ elements.  All the graphs (undirected, directed and bipartite) discussed here will be simple labeled graphs, i.e., without  multiple edges or self-loops. The degree sequence $\bd(G)$ of a graph $G=(V,E)$ is the  sequence of its vertex degrees: $\bd(G)_i= d(v_i).$  A non-negative integer sequence $\bd= (d_1,\ldots, d_n)$ is {\em graphical} iff $\bd(G) = \mathbf{d}$ for some simple graph $G$,  in which case $G$ is said to be a {\em graphical realization} of $\mathbf{d}$.  $K_n$ will denote the complete graph on $n$ vertices and $K_{n,m}$ the complete bipartite graph between sets with $n$ and $m$ vertices, respectively.

Let $G$ be a simple graph and assume that $a,b,c$ and $d$ are distinct vertices.  Furthermore, assume that $(a,c), (b,d) \in E(G)$ while $(b,c), (a,d) \not \in E(G)$. Then
\begin{equation}\label{eq:swap}
E(G')= E(G) \setminus \{(a,c), (b,d)\} \cup \{(b,c), (a,d)\}
\end{equation}
is another realization of the same degree sequence. We call  such operation a {\em swap} (it  is also called a ``switch'' or a ``2-switch'' in the literature) and denote it by $ac, bd \Rightarrow bc, ad$  (the notation implies
that $(b,c)$ and $(a,d)$ were non-edges before the swap). Note that $ac, bd \Rightarrow ab, cd$ is another swap.

The swap operation allows to treat the space of all graphical realizations of a given degree sequence as a graph $\bG(\bd)$ itself: the ``vertices'' of $\bG(\bd)$ are the graphical realizations $G \in V(\bG)$ and two graphical realizations
$G,H \in V(\bG)$ are connected by an edge in $\bG$ if a swap takes one realization into the other.

\medskip\noindent
 Similar notions  can be defined for bipartite graphs. If $B$ is a simple bipartite graph then its vertex  classes/partitions will be denoted by $U(B)=\{ u_1,\ldots, u_k\}$ and $ W(B)=\{w_1,\ldots,w_\ell\}$,  respectively,  with $V(B)=U(B)\cup W(B)$.  The {\em bipartite degree sequence} of $B$, $\bbd(B)$ is defined via:
$$
\bbd(B) = \Big(\big (d(u_1), \ldots,d(u_k)\bigr ), \bigl (d(w_1),\ldots,d(w_\ell)\bigr )\Big )  = \left( \mathbf{d}(U), \mathbf{d}(W) \right).
$$
We can define the swap operation for bipartite realizations similarly to (\ref{eq:swap}) but we must take some care: it is not enough to assume that $(b,c), (a,d) \not \in E(G)$ but we also have to make sure that $a$ and $b$ are in one vertex class and $c$ and $d$ are in the other.

To make clear whether a vertex pair can form an edge in a realization or not (because the edge would be forbidden
for some reason) we will call a vertex pair a  {\em chord} if it could hold an actual edge in a realization. Those pairs which cannot accommodate an edge are the {\em non-chords}. For example, pairs from the same vertex class of a bipartite graph are non-chords.

\medskip\noindent
For directed graphs we consider the following definitions: Let $\vec G$ denote
a simple directed graph (no parallel edges, no self-loops, but oppositely directed edges between two vertices are allowed) with vertex set $X(\vec G) = \{x_1,x_2,  \ldots,x_n \}$ and edge set $E( \vec G)$. We use the bi-sequence
$$
\bdd(\vec G) =  \left( \mathbf{\a}, \mathbf{\b} \right)
$$
to denote the sequence of degrees, where $\mathbf{\a}$ stands for the out-degree sequence (i.e., $\a(x_1),$ $\ldots,\a(x_n)$ ) and $\mathbf{\b}$ for in-degrees. A bi-sequence of non-negative integers is called a {\em graphical directed degree sequence} if there exists a simple  directed graph $\vec G$ such that $\mathbf{(\a,\b)}=\bdd(\vec G)$. In this case we say that $\vec G$ {\em  realizes}   $\mathbf{(\a,\b)}$.

We will apply the following {\em representation} of  the directed graph $\vec G$ (Gale, 1957):  let $B({\vec G})=(U,W; E)$ be  a bipartite graph where each class consists of one copy of every vertex from $V(\vec G)$.  The edges adjacent to a vertex $u_x$ in class $U$ represent the out-edges from $x$, while the edges adjacent to a vertex $w_x$ in class $W$ represent the in-edges to $x$ (so a directed edge $xy$ is identified with the edge $u_xw_y$).  Since there is no self-loop in our directed graph,  there is no $(u_x,w_x)$ type edge in its bipartite realization - these vertex pairs are non-chords,
i.e, forbidden edges.

\medskip\noindent
The {\em restricted bipartite degree sequence} problem $\bbd^{\F}$ consists of a bipartite degree sequence $\bbd$ on $(U,W)$, and a set $\F \subset [U,W]$ of {\em forbidden} edges (i.e., non-chords). The problem is to decide whether there is a bipartite  graph $G$ on $(U,W)$ completely avoiding the elements of $\F$ such that it realizes the given  bipartite degree sequence. Clearly, the bipartite representation of directed graphs is a particular bipartite restricted degree sequence problem  with $\F$ a forbidden 1-factor (a not necessarily perfect matching), i.e., forbidden $(u_x,v_x)$ type edges. This problem class was introduced in paper \cite{EKMS},  along with a Havel-Hakimi type graphicality test for restricted bipartite degree sequences.

Similarly to undirected degree sequences, one can also define the corresponding realization graphs for bipartite degree sequences ($\bG(\bbd)$), directed degree sequences ($\bG(\bdd)$) and restricted bipartite degree sequences ($\bG(\bbd^{\F})$).

\subsection{Markov Chain Monte Carlo sampling}\label{subsec:MCMC}

 For an in-depth review on general MCMC sampling and mixing times see \cite{LPW08}. The standard Markov chain for graph sampling is a weighted random walk on the realization graph $\bG$ and it is an irreducible, aperiodic and reversible chain. Typically it is chosen to be a lazy chain \chg{\cite{LPW08}}, so that bounding the mixing time reduces to the analysis of the second largest eigenvalue $\lambda_2$ of its transition matrix, or equivalently of its spectral gap $1-\lambda_2$. Here we will only consider lazy chains.  Accordingly, the chain is fast mixing iff the relaxation time $(1-\lambda_2)^{-1} = O(poly(n))$, where $n$ is the length of the degree sequence (the number of vertices of the graphs realizing the sequence).(See  Sinclair, \cite[Theorem 5]{S92}.) Since we consider realization graphs  $\bG$ for  MCMC sampling only, they will be referred to here as {\em Markov graphs}.

It is well known that the space (i.e., the set) of all simple realizations of a graphical degree sequence is connected via swap operations, which implies that the corresponding swap-based Markov Chain is irreducible, and the same applies for bipartite graphs as well. For directed graphs an analogous result holds. For the bipartite representation
of directed graphs (i.e., with the forbidden 1-factor) the usual swap definition applies between pairs of vertices that form chords. In this case we call the operation a {\em $C_4$-swap}. However, the following operation is also valid: assume that in a realization $B({\vec G})$ $(u_1,v_4), (u_2, v_5)$ and $(u_3,v_6)$ are edges, $(u_1,v_5), (u_2,v_6)$ and $(u_3,v_4)$ are non-edges but chords (with $u_i \in U$ and $v_j \in W$) and finally, the other three vertex pairs are forbidden (they belong to $\F$). Then we allow the so-called $C_6$-swap \cite{EKM}: we exchange the first three with the second three: $u_1v_4,u_2v_5, u_3v_6 \ \Rightarrow \ u_1v_5,u_2v_6, u_3v_4$. This was first introduced by Kleitman and Wang (1973, \cite{KW73}) then also by Erd\H{o}s, Mikl\'os and Toroczkai (2009, \cite{EMT09}). As Greenhill pointed out \cite{G11}, in case of regular directed degree sequences the $C_6$-swaps are not necessary.

\medskip\noindent
In 1999 Kannan, Tetali and Vempala \cite{KTV97} conjectured (referred to here as the KTV conjecture) that the swap-based MCMC is {\em rapidly mixing}, i.e.,  a pseudo-random realization is achieved after polynomially many steps
in the number of vertices (length of the degree sequence).  While this conjecture is still open, there have been a series of partial results  obtained over the years for specific degree sequence classes. In the next theorem we summarize those earlier results that play a role in the present work:

\medskip
\begin{theorem}\label{th:MCMCk}
The swap Markov chain mixes rapidly for the following degree sequences:
\begin{enumerate}[{\rm (A)}]
\item $\bd$ is regular degree sequence of simple graphs: Cooper Dyer and Greenhill {\rm \cite{CDG07, CDG12}}.
\item $\bd$ is a regular directed degree sequence: Greenhill {\rm \cite{G11}}, only $C_4$-swaps are needed.
\item $\bd$ is {\em half-regular} bipartite degree sequence: Mikl\'os, Erd\H{o}s and Soukup {\rm \cite{MES}}. Half-regularity means, that in one class the degrees are the same (i.e., regular), while in the other, the only restrictions are those imposed by graphicality.
\item $\bd$ a graphical sequence with the property that the maximum degree satisfies $3 \le d_{\max} \le \frac{1}{4}\sqrt{M}$, where $M$ is the sum of the degrees: Greenhill {\rm \cite{G15}}.
\item $\bd$ belongs to an {\em almost-regular} graph, or an {\em almost-half-regular} bipartite graph\/: Er\-d\H{o}s, Mikl\'os and Toroczkai {\rm \cite{MCMC-JDM}}. Here almost-regular means that for any degree pair $|d(v_1) - d(v_2)| \le 1.$ The meaning of almost-half-regular is analogous.
\item $\bd = \bd^{\F}$ is a restricted half-regular bipartite degree sequence where $\F$ is a (partial) matching: Erd\H{o}s, Kiss, Mikl\'os and Soukup {\rm \cite{EKMS}}. The process uses $C_4$- and $C_6$-swaps, therefore, while it contains the directed degree sequence problem as a special case, it is not comparable with the result in {\rm (B)}.
\end{enumerate}
\end{theorem}

\medskip\noindent
There are other degree sequence classes for which the swap Markov chain is clearly fast mixing. For example, the so-called {\em threshold} degree sequences \cite{CH} have exactly one realization and thus their Markov chain is trivial.  In the analogous case of threshold graphs for bipartite sequences their realization is called  a {\em difference graph} and was introduced in \cite{HPS}. In general, if there are only a small number of possible realizations of a degree sequence, then the corresponding swap Markov chain is fast mixing:
\begin{lemma}\label{th:kicsi0}
When the  number of possible realizations is polynomial in the size of the bipartite degree sequence, then the corresponding Markov chain is fast mixing.
\end{lemma}

\medskip
Later on, we are going to compose a larger degree sequence on $n$ vertices from much smaller degree sequences, each on ${\cal O}(\sqrt{\log n})$ vertices. The following result will be useful in this direction:
\medskip
\begin{lemma}\label{th:kicsi}
Let $\bbd$ be a graphical bipartite degree sequence on $\sqrt{\log n}+\sqrt{\log n}$ vertices ($\sqrt{\log n}$ vertices on each side). Then the second largest eigenvalue $\lambda_2$ of the lazy swap Markov chain satisfies
\begin{equation}\label{eq:cheeger}
\frac{1}{1-\lambda_2} = O(n^2\log^4(n)).
\end{equation}
\end{lemma}
\begin{proof}
The number of possible labeled bipartite graphs on $k+k$ vertices is $2^{k^2}$. (Each labeled vertex pair may form an edge independently of others.) Therefore, the number of realizations of a given bipartite degree sequence is (much) less than $2^{k^2}$. When $k=\sqrt{\log(n)}$, then $2^{(k^2)} = n$. If a swap Markov chain contains $n$ vertices (here a vertex is a graphical realization), then the probability of any subset of states in the equilibrium distribution cannot be smaller than $\frac{1}{n}$. The transition probabilities are $\left[{O(log^2(n))}\right]^{-1}$, and thus the {\em conductance} is $\left[O(n \log^2(n))\right]^{-1}$. Equation (\ref{eq:cheeger}) then follows from the Cheeger inequality \chg{\cite{CHEEGER},\cite{LPW08}}.
\end{proof}

\medskip\noindent
As mentioned earlier, all the proofs in Theorem \ref{th:MCMCk} use Sinclair's multicommodity flow method and require a complex and technical reasoning. Another approach was used in \cite{MCMC-JDM} where the fast mixing nature of the Markov Chain under investigation was inferred from the fast mixing nature of several well known ``smaller'' chains. In other words that Markov Chain was decomposed into smaller Markov Chains with known ``good'' properties. This result will be crucial for our purposes and thus we quote it here:

\begin{theorem}[Erd\H{o}s, Mikl\'os, Toroczkai 2015, \cite{MCMC-JDM}]
\label{th:directproduct}
Let ${\cal M}$ be a class of lazy Markov chains whose state space is a $K$ dimensional direct product of spaces, and the problem size of a particular chain is denoted by $n$. Here $n$ is not bounded but we assume that $K=O(poly_1(n))$. We  also assume that
\begin{enumerate}[{\rm (1)}]
\item Any transition of the Markov chain $M \in {\cal M}$ changes only one coordinate (each coordinate with equal probability). The transition probabilities do not depend on the other coordinates.
\item The transitions on each coordinate form irreducible, aperiodic Markov chains (denoted by $M_1, M_2, \ldots M_K$), which are reversible with respect to their stationary distribution $\pi_i.$
\item Furthermore, each of $M_1, \ldots M_K$ are  rapidly mixing, i.e., with the relaxation time $\frac{1}{1-\lambda_{2,i}}$ being bounded by a $O(\mathrm{poly}_2 (n))$ for all $i$. {\rm (}As usual, $\lambda_{2}$ denotes the second largest eigenvalue of the corresponding chain.{\rm )}
\end{enumerate}
Then the Markov chain $M$ converges rapidly to the direct product of the $\pi_i$ distributions, and the second largest eigenvalue of $M$ is
$$
\lambda_{2,M} = \frac{K-1+\max_{i}\left\{ \lambda_{2,i}\right\}}{K}
$$
and thus the relaxation time of $M$ is also polynomially bounded:
$$
\frac{1}{1- \lambda_{2,M}} =
K \, O(\mathrm{poly}_2(n))=
O(\mathrm{poly}_1(n) \mathrm{poly}_2(n)). \qquad \Box
$$
\end{theorem}
\hspace{-4pt}The important property to be checked is condition (1): whenever we make a move on $M$, the movement must be entirely within one of the factor spaces.

\section{Canonical (de)compositions of degree sequences}\label{sec:decomp}

In this section we first recall the notion of {\em canonical degree sequence decompositions} introduced by Tyshkevich in \cite{T80,T00}. We also review
some of the recent results on canonical decompositions introduced by Barrus \cite{B15} that are essential for this study.

A $G=(V,E)$ graph is a {\em split graph} if its vertices can be partitioned into a clique and an independent set.
Split graphs were introduced by F\"oldes and Hammer (\cite{FH77}).
We will use the notation $V=\langle U, W \rangle$ implying that $G[U]$ is a clique while $G[W]$ is the edge-less graph. Since it is important to specify which partition is on which side (especially for later purposes) this notation is ``non-commutative'', i.e., the elements are not interchangeable. Note that a split graph may have more than one partition into a clique and an independent set (for example, if a node in the clique has no edges to any of the nodes in the independent set, it can be moved to the latter). We will call $U$ the {\em primary class} and $W$ the {\em secondary class}. Either classes can be empty but not both, simultaneously.

Split graphs are recognizable from their degree sequences: from the Erd\H{o}s-Gallai theorem on degree sequences it follows that:
\begin{theorem}[Hammer and Simeone, 1981 \cite{HS}, Tyshkevich {\em et al.} \cite{T81}]\label{th:HS}
Assume $d(v_1) \ge \ldots \ge d(v_n)$  and let $m$ be the largest value of $i$, s.t $d(v_i)\ge i-1.$ Then $G$ is a split graph if and only if
$$
\sum_{i=1}^m d(v_i) = m(m-1) + \sum_{i=m+1}^n d(v_i).
$$
\end{theorem}
\begin{remark}\label{th:name}
Based on this theorem it is clear that if a degree sequence $\bd$ can be realized as a split graph, then all realizations of $\bd$ are split graphs as well.
\end{remark}

\noindent
Therefore such degree sequence is called  {\em split degree sequence}.

\smallskip
Recall that any two realizations of a degree sequence are connected by a series of swap operations. Now, in a split graph the only edge pairs that can be used for such swaps are those between $U$ and $W.$ (Involving other edges would lead to multiple edges between some node pairs after the swap.) Thus, the resulting edge pair will also be running between $U$ and $W.$ This gives a second proof for Remark \ref{th:name}.

As a consequence, one can write the degree sequence in the form of $(\bu, \bw)$ where both vectors are in non-increasing order. In the following we will use the same notational expression $\langle U,W \rangle$ for our split graphs as well.

Let $(\langle U, W\rangle ; E)$ be a split graph and $\G$ an arbitrary graph. Following Tyshkevich, we define the  {\em composition} graph $\h=(\langle U, W\rangle ; E) \circ \G$ as follows: $\h$ consists of a copy of $(\langle U, W\rangle ; E)$ and a copy of $\G$ and of all the possible new edges $(u,x)$ where $u \in U, x \in V(\G).$ (The first operand in this notation is always a split graph.) Note that the composition operation above is {\em non-commutative}. The degree sequence of the composition graph is therefore:
\begin{equation}\label{eq:compsec}
 \bd(U) \oplus |V(G)|, \bd(W), \bd(V(\G))\oplus |U|
\end{equation}
where $(\bd \,\oplus\,c)$ denotes an operation in which every component of a vector $\bd$ is increased by the amount $c$. Therefore, we have also defined the composition operation between a split degree sequence and a general degree sequence.

It is easy to see that if $\G$ is a split graph $(\langle X, Y \rangle ; F)$, then the result of the composition with a split graph  is also a split graph
$\left (\langle U \cup X, W \cup Y \rangle ; E\cup F \cup E(K_{ U,X \cup Y}) \right ).$
Here we used the notation $K_{U,X \cup Y}$ which is the complete bipartite graph with vertex classes $U$ and $X \cup Y.$

The graph $\G$ is {\em decomposable} if there exist a split graph $(\langle U,W\rangle ; E)$ and a graph $\h$ such that the composition of these components $(\langle U,W\rangle ; E)\circ \h$ is $=\G.$ As the following result shows the decomposability is a property of the degree sequence rather than that of the graph itself.
\begin{theorem}[Tyshkevich \cite{T00}, \chg{Theorem 2, pp. 208-211}]\label{th:T}
{\rm (I)} The graph $\G$ with non-increasing degree $n$-sequence is decomposable iff $\exists p, q$ \chg{non-negative} integers s.t.
\begin{equation}\label{eq:decomp}
0< p+q < n, \quad \sum_{i=1}^p d_i =p(n -q -1) + \sum_{i=n-q+1}^n d_i.
\end{equation}
{\rm (II)} Call a pair $(p,q)$ satisfying condition \chg{(\ref{eq:decomp})} {\em good}. \chg{To every good pair $(p; q)$ we can associate} the decomposition
$(\langle U,W\rangle ; E)\circ \h =\G$ where $(d_1,\ldots, d_p); (d_{p+1},\ldots, d_{n-q})$ and $(d_{n-q+1},\ldots, d_n)$ are the degree sequences in $U, V(\h)$ and $W$ respectively. Moreover, every such decomposition is associated with some good pair.\\
{\rm (III)} Let $p_0$ be the minimum first component of the good pairs. Let $q_0= |\{i: \chg{d_i}< p\}|$ if $p_0\ne 0$ and $q_0=1$ otherwise. Then $(\langle U,W\rangle ; E)$ is indecomposable if and only if the associated good pair is $(p_0,q_0)$.
\end{theorem}
\begin{corollary}[Graph decomposition theorem, Tyshkevich \cite{T00}]\label{th:T1}
Every graph $\G$ can be uniquely decomposed (up to isomorphism) into the form
\begin{equation}\label{eq:decomp1}
\G=(\langle U_1, W_1\rangle; E_1) \circ \cdots \circ (\langle U_\ell , W_\ell \rangle; E_\ell)\circ \G_0
\end{equation}
where each split graph and the non-split simple graph $\G_0$ (if it exists) are indecomposable. The composition operation is associative but not commutative.
\end{corollary}
Since the composition of graphs corresponds to the composition of degree sequences, the above can be reworded to statements involving degree sequences only. \chg{Some examples of (de)compositions are found in \cite{BW}-\cite{BD15} and \cite{T00}.}

\medskip\noindent
The next two statements play crucial roles later on. In the first, we use a slightly different wording than the original result:
\begin{lemma}[Barrus and West \cite{BW}]\label{th:component-wise}
\begin{enumerate}[{\rm (i)}]
\item In the composition graph $(\langle U,W\rangle; E) \circ \G$ any swap operation belongs completely to exactly one component. In other words, all four participating vertices are within $U\cup W$ or in $V(\G).$
\item Any possible swap operation in an arbitrary simple graph $\G$ is within exactly one component of its canonical decomposition.
\end{enumerate}
\end{lemma}
\noindent For (i), if at least one vertex is from $G$ and at least one
from $U\cup W$ then there will not be a valid swap due to \chg{$K_{U,V(\G)}$, $K_{U}$} and the fact that there are no edges between $\G$ and $W$. Statement (ii) follows by simple induction. This implies that if we perform a swap operation, then we can identify the component where the swap actually happened.
\begin{theorem}[Barrus 2015 {\cite[Theorem 6]{B15}}]\label{th:direct}
Let $S=(\langle U,W\rangle ; E )$ be a split graph and $\G$ an arbitrary graph with degree sequences $\bd(S)$ and $\bd(\G)$. Then the Markov graph of the composition degree sequence (\ref{eq:compsec}) is the Cartesian (also called box- or direct-) product of the two original Markov graphs:
$$
\bG \Big (\bd(S) \circ \bd(\G) \Big ) = \bG(\bd(S)) \ \square \ \bG(\bd(\G)).
$$
\end{theorem}
\hspace{-4pt}Recall that the Cartesian product of two graphs is a graph on $V(G) \times V(H)$ where $(u,u')$ and $(v,v')$ is adjacent iff $u=v$ and $u'v' \in E(H)$ or $u'=v'$ and $uv \in E(G).$ The central theme of our paper is the following:

\medskip\noindent
{\sc Meta-theorem}: {\em If the components of the canonical decomposition of a degree sequence $\bd$ have fast mixing MCMC sampling processes, then the same applies for $\bd.$}
\begin{proof}
By Lemma \ref{th:component-wise} and Theorem \ref{th:direct} our Theorem \ref{th:directproduct} clearly applies.
\end{proof}

\medskip\noindent
There are essentially two ways to use this approach. On one hand it is possible to seek out the canonical decomposition of the degree sequence under investigation and apply the meta-theorem whenever  possible. However, on the other hand, it seems to be much more powerful to build up good degree sequences from already studied good degree sequence classes. ``Good'' sequence here means that the swap MCMC mixes fast over its realizations. The simplest one is to take several (at most polynomially many) regular split graphs and one ``good'' simple graph (such as an almost regular one or one with low maximal degrees, i.e., condition (D) of Theorem \ref{th:MCMCk}) and take their composition. This construction alone significantly multiplies the number of degree sequences with fast mixing MCMC sampling processes.

In the next Section we expand the split graph (de)composition approach to bipartite and directed graphs.

\section{Canonical decomposition of bipartite and directed degree sequences}\label{sec:bipartite}

There is a natural correspondence between split graphs and bipartite graphs, which will be heavily exploited here.
A split graph $(\langle U , W \rangle ; E)$ naturally generates a bipartite graph as the edge-induced subgraph by
the edges between the sets $U$ and $W$. We will refer to bipartite graphs generated this way as  ``splitted"
bipartite graphs, although they are not split graphs in general (at least not with the same primary and secondary
sets of vertices). For convenience we will use fracture letters for splitted bipartite graphs and splitted bipartite degree sequences.
\begin{definition}
A {\em splitted bipartite graph} $\mathfrak{B} = (\langle \mfU , \mfW \rangle ; \mfE)$ is a bipartite graph with the same vertex partitions as the {\em primary} $\mfU$ and {\em secondary} $\mfW$ vertex classes of the corresponding split graph. \chg{The edge set $\mfE$ consists of all edges from $E$ which are between the vertex classes.} In the primary class of $\mathfrak{B}$ there may be vertices of zero degree. The {\em splitted bipartite degree sequence} $\bbd \langle \mfu , \mfw \rangle$ is defined analogously. Consequently $\mfu$ may contain zeros.
\end{definition}
\begin{lemma}\label{th:biject}
There exists a natural one-to-one correspondence $\Psi$  among split graphs and splitted bipartite graphs. Similarly, there is a natural bijection between split degree sequences and splitted bipartite degree sequences.
\end{lemma}
\begin{proof}
Consider the split graph $(\langle U, W\rangle ; E)$ and delete all edges within $U$. We obtained a splitted bipartite graph
$$
\Psi\Big ((\langle U, W\rangle ; E)\Big ):= (\langle \mfU , \mfW \rangle ; \mfE),
$$
where $\mfE = E \setminus K_U$ and where $\mfU$ may contain vertices of degree zero. The other direction is self-evident. For degree sequences we define accordingly: from the split degree sequence $\bbd(\bd(U),\bd(W))$ we derive $\bbd\langle  \mfu, \mfw \rangle$ by $\mfu = \bd(U) \ominus (|U|-1)$ and $\mfw = \bd(W).$ Here $(\bd \,\ominus\,c)$ denotes an operation in which every component of a vector $\bd$ is lowered by the amount $c.$\end{proof}
\begin{definition}\label{def:bipart-composite} \chg{Let}
$\mfG_1:=(\langle \mfU , \mfW \rangle ; \mfE)$ and $\mfG_2:=(\langle \mathfrak{X} , \mathfrak{Y}\rangle; \mathfrak{F})$ be two splitted bipartite graphs. \chg{We then} define their {\em bipartite composition} through Lemma \ref{th:biject}:
\begin{equation}
\mfG_1 \circ \mfG_2 := \Psi \Big ( \Psi^{-1}(\mfG_1) \circ \Psi^{-1}(\mfG_2) \Big ). \end{equation}
This notion \chg{helps to provide} a decomposition of a splitted bipartite graph in a \chg{similar form} to equation (\ref{eq:decomp1}).
\end{definition}
Lemma \ref{th:biject} and Definition \ref{def:bipart-composite} are very useful, since they provide a full battery of tools to handle the compositions of splitted bipartite graphs. But it is clear that $\mfG_1 \circ \mfG_2$ can be constructed directly. Namely
\begin{equation}\label{eq:bieq}
\mfG_1 \circ \mfG_2 =\Big ( \langle \mfU \cup \mathfrak{X}, \mfW \cup \mathfrak{Y}\rangle ; \mfE \cup \mathfrak{F} \cup K_{\mfU, \mathfrak{Y}} \Big ),
\end{equation}
that is the new primary class is the union of the two original primary classes, the new secondary class is the union of the original secondary classes \chg{and finally,} the edge set is the union of the original edge sets, supplemented by the edge set of the complete bipartite graphs with classes $\mfU$ and $\mathfrak{Y}.$

\chg{
To illustrate these notions we give here two examples, using degree sequences. We adopt the following convention: the primary classes are in the bottom rows and the secondary ones are in the top rows. For simplicity, in the examples below, we do not reorder the vertices in the vectors along the operations, therefore on the RHS the vertices in the top rows are listed in different orders (but they represent the same bipartite degree sequence).
\begin{eqnarray}
\left [ \begin{array}{cc}
1 & 1 \\ 1 & 1
\end{array} \right ]
\circ
\left [ \begin{array}{ccc}
2 & 2 & 1 \\ 3 & 1 & 1
\end{array} \right ]
& =&
\left [ \begin{array}{ccccc}
1 & 1 & 4 & 4 & 3 \\ 4 & 4 & 3 & 1 & 1
\end{array} \right ]\label{eq:egy}
\\
\left [ \begin{array}{ccc}
3 & 1 & 1 \\ 2 & 2 & 1
\end{array} \right ]
 \circ
\left [ \begin{array}{ccc}
1 & 1 \\ 1 & 1
\end{array} \right ]
&=&
\left [ \begin{array}{ccccc}
3 & 1 & 1 & 4 & 4 \\ 4 & 4 & 3 & 1 & 1
\end{array} \right ]\label{eq:ketto}
\end{eqnarray}
Notice, that the degree sequences on the \chg{right-hand} side are the same, and the reasons for why this is the case will be made clear after Theorem \ref{th:bi-decomp}}.

It is important to \chg{emphasize} that the analogy between the composition of split graphs on one hand and composition of splitted bipartite graphs on the other hand is not complete. Any split graph can be composed \chg{with} any simple graph $G$, and if graph $G$ is, by chance, another split graph, then the composition is also a split graph. We have nothing similar in case of splitted bipartite graphs. \chg{Note} that we did not \chg{introduce the} composition of a splitted bipartite graph with a (simple) bipartite graph, \chg{as} the operation cannot be carried out without fixing the primary and secondary classes.

\chg{Also note}, that {\em any} bipartite graph can be considered as a splitted bipartite graph once we choose which vertex class is the primary one, therefore any analogue of Theorem \ref{th:HS} is meaningless in this setup.

The next result is a direct analogue of Theorem \ref{th:T}. (Recall that we do not allow edge-less graphs here):
\begin{lemma}\label{th:Tb}
The splitted bipartite graph $(\langle \mfU , \mfW \rangle ; \mfE)$ with non-increasing bipartite degree sequence  $\bbd(\mfu,\mfw)$ is decomposable iff $\exists p, q$ integers s.t.
\begin{equation}\label{eq:decomp2}
0< p< |\mfU|, 0< q < |\mfW|, \quad \sum_{i=1}^p \mfu_i =p q + \sum_{i=q+1}^{|\mfW| } \mfw_i.
\end{equation}
\end{lemma}
We do not formalize the analogues of parts (II) and (III) of Theorem \ref{th:T}. Those results are simple consequences of the correspondence $\Psi,$  similarly to the following decomposition theorem:
\begin{theorem}\label{th:bi-decomp}
Any bipartite  graph with fixed designations of its primary and secondary classes has a  unique canonical decomposition into splitted bipartite graphs (and the corresponding bipartite degree sequence into splitted bipartite degree sequences).
\end{theorem}
\begin{proof}
The simplest possible treatment is to embed the graph into a split graph, using Lemma \ref{th:biject}, then generate its canonical decomposition. Finally the components can be stripped down into splitted bipartite graphs.
\end{proof}

\medskip\noindent
\chg{
The reason that the RHSs of equations (\ref{eq:egy}) and (\ref{eq:ketto}) are the same is that on the left-hand sides not all components are undecomposable. In fact, the canonical decomposition of the right-hand side of (\ref{eq:egy}) is the following:
\begin{equation}\label{eq:harom}
\left [ \begin{array}{cc}
1 & 1 \\ 1 & 1
\end{array} \right ]
\circ
\left [ \begin{array}{c}
1  \\ 1
\end{array} \right ]
\circ
\left [ \begin{array}{cc}
1 & 1 \\ 1 & 1
\end{array} \right ]
\end{equation}
It is easy to recognize, that in the LHS of (\ref{eq:egy}) the second component is the composition of the second and third components in (\ref{eq:harom}). In the LHS of (\ref{eq:ketto}) the first component is the composition of the first and second  components in (\ref{eq:harom}). Therefore the right-hand sides of  (\ref{eq:egy})  and (\ref{eq:ketto}) must be the same.
}
\medskip

\begin{remark}
While the previous observation may provide some decomposition method for graphs using splitted bipartite graphs, we are not interested here in such a process. Instead, we will use the composition process to provide large classes of general degree sequences with fast mixing swap MCMC.
\end{remark}

\medskip\noindent It is easy to see that the analogous statement of Lemma \ref{th:component-wise} remains valid for splitted bipartite graphs as well:
\begin{lemma}\label{th:bi-component-wise}
In the composition graph $(\langle \mfU, \mfW \rangle; \mfE) \circ \G$ any swap operation belongs completely to exactly one component. In other words all four participating vertices are within $\mfU\cup \mfW$ or in $V(\G).$
\end{lemma}

The  following statement is a direct analog of Theorem \ref{th:direct}:
\begin{theorem}\label{th:bi-direct}
Let $\mfS=(\langle \mfU,\mfW\rangle ; \mfE )$ be a splitted bipartite graph with splitted bipartite degree sequence $\bbd(\mfS)$ and let $\G$ be an arbitrary graph with degree sequence $\bd(G)$. Then the Markov graph of the composition degree sequence is the Cartesian product of the two original Markov graphs:
$$
\bG \big ( \bbd(\mfS) \circ \bd(\G) \big ) = \bG\big ( \bbd(\mfS)\big ) \ \square \ \bG \big (\bd(\G) \big ).
$$
\end{theorem}
\noindent We now turn to directed graphs and directed degree sequences by first recalling our bipartite representation
$B({\vec G})=(U,W; E)$ of a directed graphs, as described in section \ref{subsec:graphs}.
It is easy to see that all the definitions and results we have obtained in Section \ref{sec:bipartite} remain almost unchanged if we repeat them for splitted bipartite degree sequences with a forbidden 1-factor. To that end we have to recognize that along the composition process the forbidden 1-factors merge into another 1-factor. Furthermore, the existence of the forbidden edges somewhat tighten the available swap operations, but they  do not affect the locality of those swaps. Most importantly the following statement is valid:

\begin{theorem}\label{th:di-direct}
Let $\mfS_1=(\langle \mfU_1,\mfW_1\rangle ; \mfE_1 )$ be a splitted bipartite graph with forbidden one factor $\F_1$ and $\mfS_2=(\langle \mfU_2,\mfW_2\rangle ; \mfE_2 )$ be a splitted bipartite graph with forbidden one factor $\F_2$ $($and with splitted bipartite degree sequences $\bbd(\mfS_i), i=1,2).$ Then the Markov graph of the composition degree sequence is the Cartesian product of the two original Markov graphs:
$$
\bG \big ( \bbd(\mfS_1) \circ \bbd(\mfS_2) \big ) = \bG\big ( \bbd(\mfS_1)\big ) \ \square \ \bG \big (\bbd(\mfS_2) \big )\,,
$$
where the forbidden 1-factor for $\bbd(\mfS_1) \circ \bbd(\mfS_2)$ is $\F_1 \cup \F_2.$
\hfill $\Box$
\end{theorem}

It is important to mention that the canonical degree sequence ``decomposition'' we use here for directed graphs has nothing to do with other decomposition methods introduced in the literature, for example by LaMar \cite{LM12}. We use this approach only to extend the class of known directed degree sequences with fast mixing Markov chains.

\section{Extending the classes of degree sequences with fast mixing swap Markov chains}\label{sec:direct}

Now we are ready to present our new degree sequence classes beyond the known ones with fast mixing MCMC sampling. At first we describe the classes, then we generate some simple estimates to compare the sizes of the old and new classes. We start with a few, almost trivial, observations:
\begin{lemma}\label{th:Mgraph}
Let $(\langle U, W\rangle ; E)$ be a split graph and $\Psi\big ((\langle U, W\rangle ; E)\big ):= (\langle \mfU , \mfW \rangle ; \mfE)$ be the corresponding splitted bipartite graph. Then the Markov graphs
\begin{equation}\label{eq:equi}
\bG\big  (\bbd(\bd(U), \bd(W))\big ) \cong \bG\big( \bbd\langle  \mfu, \mfw \rangle \big )
\end{equation}
are isomorphic under $\Psi$ and the corresponding edges have equal weights (i.e., transition probabilities). Consequently if the MCMC process is fast mixing on the second Markov graph then it is also fast mixing on the first  one.
\end{lemma}
\begin{proof}
Each split graph has many more edges than its equivalent splitted bipartite graph, namely, by the number of edges in the complete graph in $U$. However, none of these edges ever participate in any swap operation (in moving along the Markov chain). All the other edges of the split graph are in one-to-one correspondence with the edges of the splitted bipartite graph. Finally the transition probabilities for these swaps are equal by definition.
\end{proof}

Note that when we delete the edges of the $K_U$ from a split graph to obtain the corresponding splitted bipartite graph we may end up with nodes in $U$ that have zero degrees. And vice-versa, adding the edges of $K_U$ to the primary vertex set of the splitted bipartite graph we may end up with way more edges than before. However, this does not affect the MCMC process. Nodes with zero degree do not participate in any swap, and the added extra edges cannot participate either.
Thus, while formally we have new Markov graphs (and new Markov chains), there is a clear-cut natural isomorphism between the original and the ``extended'' Markov graphs and under this isomorphism the transition probabilities are completely unchanged. However, the large size of the new classes of sequences with fast mixing swap MCMC is not due to this trivial addition of nodes with zero degrees; it is already a property of the constructed class of sequences that contain no zero degree nodes and that is how we will formulate our results, below.

\medskip\noindent
Next we will carry out the program expressed in the Meta-Theorem. At first we will consider almost-half regular splitted bipartite graphs and their splitted graph equivalents (or more precisely the corresponding degree sequences) as the main building blocks of our degree sequences.

Consider $k$ almost-half-regular bipartite graphs that is, with one side of the partition the degrees of any two vertices differ by at most unity and there are no restrictions on the other side.  For every bipartite graph assign the primary $\mfU$ and secondary $\mfW$ designations to its vertex classes, thus defining a sequence $\mfS_1,\ldots,\mfS_k$ of splitted bipartite graphs. Note that it does not matter if the half-regular partition is primary or secondary, both assignments are valid and thus  the $k$ bipartite graphs generate $2^k$ sequences of splitted bipartite graphs. Furthermore let $S_1,\ldots, S_k$ denote their split graph counterparts
(here all classes $U_i$ form complete graphs) and finally, let $\bd_0$ be a degree sequence with fast mixing MCMC sampling process on its Markov graph (e.g., any degree sequence listed in Theorem \ref{th:MCMCk}). Then:
\begin{theorem}\label{th:fast1}
The degree sequences
\begin{eqnarray}
\label{eq:pld1} \bbd(S_1) &\circ& \bbd(S_2) \circ \cdots \circ \bbd(S_k) \circ \bd(\G_0); \\
\label{eq:pld2} \bbd(\mfS_1) &\circ& \bbd(\mfS_2) \circ \cdots \circ \bbd(\mfS_k) \circ \bd(\G_0)
\end{eqnarray}
have fast mixing MCMC sampling processes on their Markov graphs.
\end{theorem}
\begin{proof}
By Theorem  \ref{th:directproduct} and Theorem \ref{th:di-direct} the Meta-Theorem applies for these setups.
\end{proof}

In words, the statement says that we can compose several almost-half-regular splitted bipartite graphs, and while the composition itself formally is not almost-half-regular by any means, all compositions admit a fast sampling MCMC, with its speed determined by the slowest mixing coordinate of the chain. As discussed above, there are at most $2^k$ such compositions possible.

It is important to emphasize that the composition of almost-half-regular splitted bipartite degree sequences is, in general, very far from being almost-half-regular and the same applies if we omit the word ``half''. Additionally, the two derived degree sequences (the compositions of split degree sequences and of splitted bipartite degree sequences) and their realizations are very different - consider for example, the sizes of edge sets. However, the Markov graphs of all realizations of the two cases are isomorphic. When graph $\G_0$ is a bipartite graph, then the resulting graph in the second case is also bipartite.

For directed degree sequences we have the following, analogous result:
\begin{theorem}\label{th:fast2}
Assume that  $\mfS_1,\ldots,\mfS_k$ are almost-half-regular splitted bipartite graphs with $\F_i$, $i=1,\ldots,k$ forbidden 1-factors. Then the degree sequence
\begin{equation}
\label{eq:pld3} \bbd(\mfS_1) \circ \bbd(\mfS_2) \circ \cdots \circ \bbd(\mfS_k)
\end{equation}
admits a fast mixing MCMC sampler on its Markov graph.
\end{theorem}
\begin{proof}
By Theorem \ref{th:directproduct}   and Theorem \ref{th:bi-direct} the Meta-Theorem applies for these setups.
\end{proof}

In the next section we will consider compositions from splitted bipartite degree sequences on $m+m$ vertices, as these bipartite sequences have the largest number of graphical realizations. It is then  important to observe that the splitted bipartite sequences and their compositions will generate split graph sequences (which are non-bipartite) that do not fall under the category (D) of Theorem \ref{th:MCMCk} \cite{G15}, and in this sense these form a novel class of irregular degree sequences
with proven fast mixing MCMC, beyond Theorem \ref{th:MCMCk} (D):

\begin{theorem}\label{th:nonGreenhill}
The split graph degree sequence generated from a graphical, splitted bipartite sequence on $m+m$ vertices with $m\geq 2$ does not obey the Greenhill condition $d_{max} \leq \frac{1}{4}\sqrt{M}$, where $d_{max}$ and $M$ are the maximum degree and the sum of degrees in the generated split graph, respectively.
\end{theorem}
\begin{proof}
The split graph degree sequence is obtained by adding all the possible edges to the primary partition of the
splitted bipartite degree sequence. Accordingly, after the augmentation, we clearly must have:
$m-1 \leq d_{max}$ and $M \leq (m-1)m +2 m^2 = 3m^2-m$ (the upper bound is  realized when the splitted bipartite graph is $K_{m,m}$). The Greenhill condition would then imply that $m-1 \leq \frac{1}{4} \sqrt{3m^2-m}$, or equivalently, that $13 m^2 - 31 m +16 \leq 0$. This, however, is clearly violated for all $m \geq 2$.
\end{proof}

\subsection{Size estimates of degree classes with fast mixing MCMC}\label{subsec:count}

\noindent How large is this simply generated class of sequences compared to the number of the almost-half-regular bipartite sequences with the same total vertex numbers? For comparison, here we will only consider bipartite sequences on $m+m$ vertices, as this is the most numerous class. An almost-half-regular bipartite graph on $m+m$ vertices with $e$ total edges has one possible degree sequence on the regular-degree vertex class. The only conditions we have on the other vertex class is that no degree can exceed the number $m$ and the sequences are arranged non-increasingly.  Therefore
\begin{lemma}\label{th:number}
The number of non-increasing, graphical, almost-half-regular bipartite degree sequences on $m+m$ vertices is
$$
2{2m \choose m} - m^2 -1 = O\left(\frac{4^m}{\sqrt{m}}\right).
$$
{\rm (}For simplicity of calculation here we allow vertices with zero degree.{\rm )}
\end{lemma}
\begin{proof}
The number of non-negative, non-increasing integer sequences on $m$ vertices with largest element at most $m$ is ${2m \choose m}$. For every such integer sequence there exists exactly one almost-half-regular degree sequence with the same sum of their degrees. This degree sequence pair will always be graphical due to the Gale-Ryser theorem.

Since we can assign the primary (secondary) roles to vertices on either side of the partition, we have a total of $2{2m \choose m}$ splitted bipartite graphs. We have to subtract the degree sequences counted twice. They are exactly those degree sequences that are almost regular on both sides of the partition. For every sum of degrees, there is exactly one such non-increasing degree sequence. As the sum of degrees might vary between $0$ and $m^2$, therefore, we have to subtract $m^2+1$ from $2{2m \choose m}$.
The asymptotic follows directly from the Stirling formula, $m! \sim \ \sqrt{2 \pi m} \left(\frac{m}{e}\right)^n$.
\end{proof}

If we take the composition of $n/m$ almost-half-regular bipartite degree sequences on $m+m$ vertices then we have slightly smaller number than we have almost-half-regular bipartite degree sequences on $n+n$ vertices (divisibility conditions implied). But we can do this for all possible $m$.  However, here the problem that might arise is that some (probably small number) of sequences will be enumerated more than once. One way of overstepping this issue is by using only indecomposable splitted bipartite degree sequences of arbitrary, but of not too big size. However the number of these objects is not known.

While the constructed class above is clearly much larger than the original class of almost-half-regular bipartite degree sequences with proven fast mixing MCMC, its size is not easily estimated. Instead, we will examine in detail another class, which, as we will show, is much larger than the set of almost-half-regular degree sequences: we will compose graphs from general splitted bipartite degree sequences of relatively small number of vertices, as a direct application of Lemma \ref{th:kicsi}.

\medskip\noindent

\chg{
Composing a large number of short splitted bipartite degree sequences into longer ones and enumerating the different composition results requires good control over the possible multiplicities to estimate the size of the resulting class - as it is clearly shown by the examples (\ref{eq:egy}),(\ref{eq:ketto}) and (\ref{eq:harom}). \\
\indent
Thus, instead, we will compose splitted bipartite degree sequences of fixed length, which have the following useful property:}
\begin{lemma}\label{th:constantunique}
Let $\mfd$ and $\mff$ be two splitted bipartite degree sequences, both are derived as compositions  of splitted bipartite degree sequences:
\begin{eqnarray*}
\mfd &=& \mfd_1 \circ \mfd_2 \circ \ldots \circ \mfd_n \quad \hbox{and} \\
\mff &=& \mff_1 \circ \mff_2 \circ \ldots \circ \mff_n,
\end{eqnarray*}
where \chg{all} $\mfd_i$ and $\mff_i$ \chg{are {\rm(}not necessarily indecomposable{\rm )}} splitted bipartite degree sequences on the same number of vertices (e.g., on $k+k$). If there exists an $i$ such that $\mfd_i \ne \mff_i$ then $\mfd \ne \mff$.
\end{lemma}
\begin{proof}
Assume, that on the contrary, $\mfd = \mff$ in spite the fact that $\mfd_i \ne \mff_i$. Consider now the canonical decomposition of $\mfd_i$ and $\mff_i$ for the smallest $i$ such that $\mfd_i \ne \mff_i$. Due to the associative rule of the $\circ$ operation, they can be written into the form
\begin{eqnarray*}
\mfd &=& \mfd_1 \circ \mfd_2 \circ \ldots \circ (\mfd_{i,1}  \ldots \circ \mfd_{i,j}) \circ \mfd_{i+1} \circ  \ldots \quad \hbox{and} \\
\mff &=& \mff_1 \circ \mff_2 \circ \ldots \circ (\mff_{i,1} \circ \ldots \circ \mff_{i,\ell}) \circ \mff_{i+1} \circ \ldots ,
\end{eqnarray*}
where $\mfd_i$ and $\mff_i$ are written in their canonical decomposition form. Since $\mfd_i \ne \mff_i$, there must be a first $k$ such that $\mfd_{i,k} \ne \mff_{i,k}$, though they are both in the same position of the decomposition for $\mfd$ and $\mff$. Thus, the canonical decompositions of $\mfd$ and $\mff$ differ, implying that $\mfd \ne \mff$, a contradiction.
\end{proof}
\begin{lemma}\label{lem:4.99}
The number of graphical splitted bipartite degree sequences that contain at most $6+6$ long splitted bipartite degree sequences in their canonical decompositions is $\Omega(4.99^n)$.
\end{lemma}
\begin{proof}
Consider the splitted bipartite degree sequences we can construct by composing graphical splitted bipartite degree sequences having exactly $6+6$ vertices. We know that there are $15584$ graphical degree sequences on $6+6$ vertices (see \cite{BR} and/or \cite{oeis}). Therefore we can construct $15586^{\frac{n}{6}} > 4.99^n$ different graphical degree sequences in this way due to Lemma~\ref{th:constantunique}. This is obviously less than all the possible cases, which proves the lemma.
\end{proof}
\begin{theorem}
Let $g(n)$ be the number of graphical splitted bipartite degree sequences on $n+n$ vertices for which we can prove rapid mixing using the Meta theorem and let $h(n)$ denote all almost-half-regular graphical splitted bipartite degree sequences on $n+n$ vertices. Then there exist a $c>1$ such that
$$g(n)/h(n) = \Omega\left(c^n\right).$$
\end{theorem}
\vspace{-10pt}
\begin{proof}
This is a direct consequence of the Lemma  \ref{th:kicsi}, applied for splitted bipartite degree sequences,  and the decomposition \chg{result Theorem \ref{th:di-direct}}.
\end{proof}

\noindent Numerical calculations shows if we consider splitted bipartite degree sequences on $25+25$ vertices then we can write $10$ instead of $4.99.$ In general, we can state the following \chg{theorem}:
\begin{theorem}
Let $C$ be a constant s.t. the number of graphical bipartite degree sequences on $n+n$ vertices is $\Omega(C^n).$ Then for any $\epsilon>0$ there exist $\Omega((C-\epsilon)^n)$ bipartite degree sequences with fast mixing MCMC processes.
\end{theorem}
\begin{proof}
If the number of graphical degree sequences on $n+n$ vertices is $\Omega(C^n)$, then there exists an $\alpha>0$ such that for any $n$, the number of graphical degree sequences on $n+n$ vertices is greater than $\alpha C^n$. Let $n_0 = \left\lceil \frac{\log(\alpha)}{\log(\frac{C-\epsilon}{C})} \right\rceil$. Then the number of graphical sequences on $n_0 + n_0$ vertices is greater or equal than $(C-\epsilon)^{n_0}$. Similarly to Lemma~\ref{lem:4.99}, we can prove that there are $\Omega((C-\epsilon)^n)$ number of graphical splitted bipartite degree sequences that contain at most $n_0 + n_0$ long splitted bipartite degree sequences in their canonical decompositions. Their swap Markov chains will be all rapidly mixing.
\end{proof}

If one could prove the following conjecture than we could prove an even slightly stronger statement.
\begin{conjecture}\label{con:logconvex}
The number of bipartite graphical degree sequences on $n+n$ vertices is a logconvex function of $n$.
\end{conjecture}
\begin{theorem}
If Conjecture~\ref{con:logconvex} is true, then for any $\epsilon > 0$, there exists a polynomial function $poly(n)$, such that
$$\frac{f(n)}{g(n)} = O((1+\epsilon)^n)$$
where $f(n)$ denotes the number of bipartite graphical degree sequences on $n+n$ vertices and $g(n)$ denotes the number of bipartite  graphical degree sequences for which the second largest eigenvalue $\lambda_2$ of their swap Markov chain satisfies
$$\frac{1}{1-\lambda_2} < \mbox{poly}(n).$$
\end{theorem}
\begin{proof}
If $f(n)$, the number of graphical degree sequences on $n+n$ vertices is logconvex, then the derivative of its logarithm has a limit as $n$ tends to infinite. This is because logconvexity means that the derivative of $\log(f(n))$ is monotonically increasing. However, $f(n) = O(16^n)$, thus the derivative is upper bounded, and any upper bounded, monotoniously increasing series has a limit. (We have $f(n) = O(16^n)$ because the number of pairs of degree sequences on $n+n$ vertices with maximum degree $n$ is $O\left(\frac{16^n}{n}\right)$ -- see Lemma \ref{th:number} -- and the number of graphical degree sequences is less.) Let
$$C= \lim_{n \rightarrow \infty} \frac{\partial}{\partial n} \log(f(n)).$$
Then for any $\epsilon'>0$, $f(n) = O((C+\epsilon')^n)$ and $f(n) = \Omega((C-\epsilon')^n)$. Since $f(n) = \Omega((C-\epsilon')^n))$, there exists an $n_0$ constant such that $f(n_0)>(C-2\epsilon')^n$. Therefore the number of graphical splitted bipartite degree sequences on $n+n$ vertices that contain at most $n_0 + n_0$ long splitted bipartite degree sequences in their canonical decompositions is $g(n)=\Omega(C-2\epsilon')^n)$. What follows is that
$$\frac{f(n)}{g(n)} = O\left(\frac{(C+\epsilon')^n}{(C-2\epsilon')^n}\right).$$
If we set $\epsilon'$ such that
$$\frac{C+\epsilon'}{C-2\epsilon'} = 1+\epsilon$$
holds, the theorem follows.
\end{proof}

Finally, it is clear that in our construction we can use almost-half-regular splitted bipartite sequences, difference graph sequences, small size splitted bipartite sequences mixed in any order to generate bipartite degree sequences with fast mixing Markov Chains.

\section{MCMC sampling on degree spectra matrix problems}\label{sec:spectra}

Let us recall that in the graph $\G$ the {\em degree spectrum} of vertex $v$ is the vector $\bs_{\G}(v)$ where $\bs_{\G}(v)_i$ denotes the number of neighbors of $v$ that have degree $i$. The {\em degree spectra matrix} $M(\G)$ consists of the degree spectra of the vertices as columns. One can ask whether an integer matrix can be the degree spectra matrix of a graph. If the answer is affirmative, then the matrix is {\em graphical}.

The degree spectrum of a given vertex automatically defines its degree, therefore this notion can be thought as a specialization of the degree sequences. Indeed, in general there are several degree spectra matrices corresponding to the same degree sequence and their individual realization sets partition the set of all realizations of the degree sequence into classes.

To our best knowledge, the notion was \chg{first} introduced in \cite{JDM} and was further studied in \cite{MCMC-JDM} to deal with different aspects of the Joint Degree Matrix problem. It is easy to decide whether a degree spectra matrix is graphical (see \cite[Theorem 3]{sampling-JDM}): For all pairs $1 \le i, j \le \Delta(\G)$ denote $\G_{i,j}$ the induced subgraph spanned degree-$i$ and degree-$j$ vertices. (Here $i=j$ can happen, in which case we have a simple graph instead of a bipartite graph.) Then
$$
\bbd(\G_{i,j}) = \big (( \bs_{\G}(u)_j : d(u)=i) (\bs_{\G}(w)_i : d(w)=j) \big )
$$
is a bipartite (simple) degree sequence.
\begin{theorem}[Bassler, Del Genio, Erd\H{o}s, Mikl\'os and Toroczkai 2015]\label{th:BDEMT}
The degree spectra matrix $M$ is graphical iff all its component (bi- and uni-)partite degree sequences are graphical.
\end{theorem}

This clearly refers the fact that the set of all realizations is connected under the swap operation, i.e., to the {\em irreducibility}  of the space of all realizations. Indeed, each swap is completely within one of the component graphs, therefore the Markov graph of all realizations is clearly partitioned into these smaller Markov Graphs. Finally, the paper also developed a polynomial time algorithm to determine all possible degree spectra matrices, which are compatible with the degree sequence of the graph.

In \cite{BD15} Barrus and Donovan reintroduced the notion of degree spectra under the name of {\em neighborhood degree list} and they reproved Theorem \ref{th:BDEMT} and also the connectedness (irreducibility) result. However, the main results in their paper are on the uniqueness of realizations (up to isomorphism) and their connections to threshold graphs.

\begin{theorem}\label{th:spectra}
Let $\bd$ be a degree sequence and assume that for a compatible degree spectra matrix $M$, all component graphs admit fast, swap-based MCMC samplers. For example, the bipartite graphs are almost-half regular and the simple graphs are almost-regular or irregular but satisfy the Greenhill condition (D) of Theorem \ref{th:MCMCk}. Then the corresponding realizations of the degree spectra matrix all admit fast mixing swap-based MCMC sampling processes.
\end{theorem}

An example for this kind of degree spectra matrix has already been found in \cite[Corollary 5]{JDM}. Namely,  that paper proved the following result:
\begin{theorem}[Czabarka, Dutle, Erd\H{o}s and Mikl\'os \cite{JDM}] \label{th:balanced}
For any graphical Joint Degree Matrix there exist degree spectra matrices for which all component graphs are almost regular or almost semi-regular.
\end{theorem}

\noindent
Recall that, a bipartite graph is {\em semi-regular} if in both classes the vertices have the same degree (but the values can be different between the partition classes). In  \cite{JDM} this is called a {\em balanced realization}. It is also clear that almost-half-regularity is a much less severe property then almost-semi-regularity. In fact, from any almost-semi-regular bipartite graph pair $\G_{i,j}$ and $\G_{i,\ell}$ one can easily make several almost-half-regular realizations with swap operations which keep the Joint Degree Matrix requirements but destroy almost-semi-regularity.

{\section{Conclusions}\label{sec:discuss}

In summary, by exploiting an earlier result obtained by us on composition Markov chains for direct-product spaces combined with the split graph decompositions introduced by Tyskevich and a recent result of Barrus and West we could significantly extend the class of bipartite degree sequences for which the KTV conjecture holds. This approach does not only contribute to the KTV conjecture but also opens up exciting novel perspectives on the intimate relationships between processes on graphs and deeper underlying graph theoretical properties.}

\goodbreak
\bibliographystyle{plain}

\begin{thebibliography}{99}
\bibitem{BW} {\sc M. D. Barrus and D. B. West}, {\em The $A_4$-structure of a graph}, J. Graph Theory 71 (2) (2012), pp. 159--175.

\bibitem{B15} {\sc M. D. Barrus} {\em On realization graphs of degree sequences}, arXiv:1503.06073v1 (2015), pp. 1--10.

\bibitem{BD15} {\sc M. D. Barrus and E. Donovan}, {\em Neighborhood degree lists of graphs}, arXiv:1507.08212v1 (2015), pp. 1--12.

\bibitem{sampling-JDM} {\sc K.E. Bassler,  C.I. Del Genio, P.L. Erd\H{o}s, I. Mikl\'os  and Z. Toroczkai}, {\em  Exact sampling of graphs with prescribed degree correlations},  New J. Phys. 17 (2015), \#083052 pp 19.

\bibitem{R6} {\sc I. Bez\'akov\'a, N. Bhatnagar and E. Vigoda}, {\em Sampling Binary Contingency Tables with a Greedy Start}, Random Structures and Algorithms, 30 (1-2) (2007), pp. 168--205.

\bibitem{Bezakova} {\sc I. Bez\'akov\'a}, {\em Sampling binary contingency tables}, Comp. Sci. Eng., 10(2)  (2008), pp. 26--31.

\bibitem{R7} {\sc I. Bez\'akov\'a, N. Bhatnagar and D. Randall}, {\em On the Diaconis-Gangolli Markov chain for sampling contingency tables with cell-bounded entries}, J. Comb. Optim. 22(3) (2011), pp. 457--468.

\bibitem{Blitz} {\sc J. Blitzstein and P. Diaconis}, {\em A sequential importance sampling algorithm for generating random graphs with prescribed degrees}, Internet Math., 6 (2011), pp. 489--522.

\bibitem{BR} {\sc R.A. Brualdi and H.J. Ryser}, {\em Combinatorial Matrix Theory}, Cambridge Univ. Press, 1992.

\bibitem{CHEEGER}\chg{{\sc J. Cheeger}, {\em A lower bound for the smallest eigenvalue of the Laplacian},  Problems in Analysis, (R. C. Gunning, ed.) Princeton Univ. Press (1970), pp. 195-199.}

\bibitem{R1}  {\sc Y. Chen, P. Diaconis, S.P. Holmes, and J.S. Liu}, {\em Sequential Monte Carlo Methods for Statistical Analysis of Tables}, Journal of the American Statistical Association, 100(469) (2005), pp. 109--120.

\bibitem{CH} {\sc V. Chv\'atal and P.L. Hammer} {\em Aggregation of inequalities in integer programming}, in Hammer,  Johnson, Korte et al., Studies in Integer Programming (Proc. Worksh. Bonn 1975), Annals of Discrete Mathematics Vol. 1, Amsterdam: North-Holland 1977, pp. 145--162.

\bibitem{CDG07} {\sc C. Cooper, M. Dyer and C. Greenhill}, {\em Sampling regular graphs and a peer-to-peer network}, Comp. Prob. Comp., 16(4) (2007), pp. 557--593.

\bibitem{CDG12} {\sc C. Cooper, M. Dyer and C. Greenhill}, {\em Corrigendum: Sampling regular graphs and a peer-to-peer network}, arXiv:1203.6111v1 (2012), pp. 8.

\bibitem{R4} {\sc  M. Cryan, M.Dyer, L.A. Goldberg, M. Jerrum and R. A. Martin}, {\em Rapidly Mixing Markov Chains for Sampling Contingency Tables with a Constant Number of Rows}, SIAM J. Comput. 36(1) (2006), pp. 247--278.

\bibitem{R8} {\sc  M. Cryan, M. E. Dyer and D. Randall}, {\em Approximately Counting Integral Flows and Cell-Bounded Contingency Tables}, SIAM J. Comput. 39(7) (2010), pp. 2683--2703.

\bibitem{JDM} {\sc \'E. Czabarka, A. Dutle, P.L. Erd\H{o}s, I. Mikl\'os}, {\em On Realizations of a Joint Degree Matrix}, Disc. Appl. Math 181 (2015), pp. 283--288.

\bibitem{SCS} {\sc C.I. Del Genio, H. Kim, Z. Toroczkai, K.E. Bassler}, {\em Efficient and exact sampling of simple graphs with given arbitrary degree sequence}, PLoS ONE, 5(4) (2010),  e10012.

\bibitem{DSC93} {\sc P. Diaconis and L. Saloff-Coste}, {\em Comparison theorems for reversible Markov Chains},  Ann. Appl. Probab., 3(2) (1993), pp. 696--730.

\bibitem{DG95} {\sc P. Diaconis and A. Gangolli}, {\em Rectangular  Arrays with Fixed Margins}. Discrete Probability and Algorithms, Eds.: D. Aldous et al., Springer-Verlag,  (1995), pp. 15--41.

\bibitem{EKM} {\sc P.L. Erd\H{o}s,  Z. Kir\'aly and I. Mikl\'os}, {\em On graphical degree sequences and realizations}, Combinatorics, Probability and Computing 22 (3) (2013), pp. 366--383.

\bibitem{EMT09} {\sc P.L. Erd\H{o}s, I. Mikl\'os and Z. Toroczkai}, {\em A simple Havel-Hakimi type algorithm to realize graphical degree sequences of directed graphs},   Elec. J. Combinatorics 17 (1) (2010),   R66 (10pp)

\bibitem{MCMC-JDM} {\sc P.L. Erd\H{o}s, I. Mikl\'os and Z. Toroczkai} {\em A decomposition based proof for fast mixing of a Markov chain over balanced realizations of a joint degree matrix},  SIAM J. Disc. Math 29 (1) (2015), pp. 481--499.

\bibitem{EKMS} {\sc P.L. Erd\H{o}s, Z. S. Kiss, I. Mikl\'os and L. Soukup}, {\em Approximate Counting of Graphical Realizations},  PLOS ONE (2015), pp 20. \#e0131300.

\bibitem{R5} {\sc T. Feder, A. Guetz, M. Mihail and A. Saberi}, {\em A Local Switch Markov Chain on Given Degree Graphs with Application in Connectivity of Peer-to-Peer Networks}, FOCS'06 (2006), pp. 69--76.

\bibitem{FH77} {\sc S. F\"oldes and P. L. Hammer}, {\em Split graphs}, Proceedings of the Eighth Southeastern Conference on Combinatorics, Graph Theory and Computing (Louisiana State Univ., Baton Rouge, La., 1977), Congressus Numerantium XIX, Winnipeg: Utilitas Math., pp. 311--315.

\bibitem{G11} {\sc C. Greenhill}, {\em  A polynomial bound on the mixing time of a Markov chain for sampling regular directed graphs}, Electronic J. Comb., 16(4) (2011), pp. 557-593.

\bibitem{G15} {\sc C. Greenhill}, {\em The switch Markov chain for sampling irregular graphs}, in Proc. 26th ACM-SIAM Symposium on Discrete Algorithms, New York-Philadelphia (2015), pp. 1564--1572.

\bibitem{Gross} {\sc E. Gross, S. Petrovic\'c and D. Stasi}, {\em Goodness-of-fit for log-linear network models: Dynamic Markov bases using hypergraphs}, Ann. Inst. Statist. Math., in press. arXiv:1401.4896v1 (2014), pp. 1--28.

\bibitem{HS} {\sc P.L. Hammer and B. Simeone}, {\em The splittance of a graph}, Combinatorica 1 (3) (1981), pp. 275--284.

\bibitem{HPS} {\sc P.L. Hammer, U.N. Peled and X. Sun}, {\em Difference graphs} Discrete Appl. Math. 28 (1990), pp. 35--44.

\bibitem{kim1} {\sc H. Kim, Z. Toroczkai, P.L. Erd\H{o}s, I. Mikl\'os and L.A. Sz\'ekely}, {\em Degree-based graph construction} J. Phys. A: Math. Theor., 42 (2009), 392001.


\bibitem{scdir} {\sc H. Kim, C.I. Del Genio, K.E. Bassler and Z. Toroczkai}, {\em Constructing and sampling directed graphs with given degree sequences} New J. Phys., 14 (2012), 023012.

\bibitem{KTV97} {\sc R. Kannan, P. Tetali, and S. Vempala}, {\em Simple Markov-Chain Algorithms for Generating Bipartite Graphs and Tournaments}, Random Structures Algorithms, 14(4) (1999), pp. 293--308.

\bibitem{KW73} {\sc D.J.~Kleitman and D.L.~Wang}, {\em Algorithms for constructing graphs and digraphs with given valences and factors}, Discrete Math. 6 (1973), pp. 79--88.

\bibitem{LM12} {\sc M. D. LaMar}, {\em Splits digraphs}, Discrete Math. 312 (2012), 1314--1325.

\bibitem{LPW08} {\sc D. A. Levin, Y. Peres and E. L. Wilmer} {\em Markov Chains and Mixing Times}  (2008), American Mathematical Society, Providence, RI.

\bibitem{MR02} {\sc R. Madras and D. Randall}, {\em Markov chain decomposition for convergence rate analysis}, Ann. Appl. Probab., 12 (2002), pp. 581--606.

\bibitem{MR06} {\sc R. Martin and D. Randall}, {\em Disjoint decomposition of Markov chains and sampling circuits in Cayley graphs}, Combin. Probab. Comput., 15 (2006), pp 411--448.

\bibitem{MES} {\sc I. Mikl\'os, P.L. Erd\H{o}s and L. Soukup}, {\em Towards random uniform sampling of bipartite graphs with given degree sequence}, Electronic J. Comb., 20(1) (2013), P16.

\bibitem{oeis} {\sc Online Encyclopedia of Integer Sequences}, {\em https://oeis.org/A029894}

\bibitem{sonja2} {\sc S. Petrovi\'c}, {\em A survey of discrete methods in (algebraic) statistics for networks},  chapter in Contemporary Mathematics (CONM) book series, American Mathematical Society, Eds. H. Harrington, M. Omar, and M. Wright; in press (2016); \url{http://arxiv.org/abs/1510.02838}.

\bibitem{rand} {\sc D. Randall}, {\em Rapidly Mixing Markov Chains with Applications in Computer Science and Physics}, Comp. Sci. Eng., 8(2) (2006), pp. 30--41.

\bibitem{rao} {\sc A.R. Rao, R. Jana and S. Bandyopadhyay}, {\em  A Markov chain Monte Carlo method for generating random $(0,1)$-matrices with given marginals}, Sankhy\={a}: Ind. J. Stat., 58 (1996), 225--370.

\bibitem{ryser} {\sc H.J. Ryser}, {\em Combinatorial properties of matrices of zeros and ones}, Canad. J. Math., 9 (1957) 371--377.

\bibitem{S92} {\sc A. Sinclair}, {\em Improved bounds for mixing rates of Markov chains and multicommodity flow}, Combin. Probab. Comput.,  1 (1992), pp. 351--370.

\bibitem{sonja1} {\sc A. Slavkovi\'c, X. Zhu and S. Petrovi\'c}, {\em Fibers of multi-way contingency tables given conditionals: relation to marginals, cell bounds and Markov bases}, Ann. Inst. Stat. Math., 67 (2015), pp. 621--648.

\bibitem{taylor} {\sc R. Taylor},  {\em Constrained switching in graphs}, in Combinatorial Mathematics VIII, Springer LNM, vol. 884, 1981, pp 314--336.

\bibitem{T80} {\sc R. Tyshkevich}, {\em Canonical decomposition of a graph}, (in Russian) Doklady Akademii Nauk BSSR XXIV 8 (1980), pp. 677--679.

\bibitem{T81} {\sc R. Tyshkevich, O. Melnikov and V. Kotov}, {\em On graphs and degree sequences}, (in Russian) Kibernetika 6 (1981), 5--8.

\bibitem{T00} {\sc R. Tyshkevich}, {\em Decomposition of graphical sequences and unigraphs}, Discrete Math. 220 (1-3) (2000), pp. 201--238.
\end{thebibliography}

\end{document}